\documentclass[11pt]{extarticle}

\usepackage[left=1in,right=1in,top=1in,bottom=1.3in]{geometry}
\usepackage{amsmath, amssymb, amsthm, latexsym, amsfonts, indentfirst, xcolor, mathtools, microtype}
\usepackage[utf8]{inputenc} 
\usepackage[english]{babel}
\usepackage[breaklinks]{hyperref}
\hypersetup{
	colorlinks = true, % Colours links instead of ugly boxes
	urlcolor = cyan, % Colour for external hyperlinks
	linkcolor = teal, % Colour of internal links
	citecolor = cyan % Colour of citations
}

\usepackage{graphicx}
\graphicspath{{./}}
\let\OLDthebibliography\thebibliography
\renewcommand\thebibliography[1]{
	\OLDthebibliography{#1}
	\setlength{\parskip}{0pt}
	\setlength{\itemsep}{0pt plus 0.3ex}
}

\newtheorem{Theorem}{Theorem}
\newtheorem{Claim}{Claim}

\newtheorem{Corollary}{Corollary}

\newtheorem{Lemma}{Lemma}
\newtheorem{Problem}{Problem}
\newtheorem{Proposition}{Proposition}

\newcommand{\R}{{\mathbb R}}
\newcommand{\Z}{{\mathbb Z}}
\newcommand{\N}{{\mathbb N}}

\newcommand{\x}{{\mathbf x}}
\newcommand{\y}{{\mathbf y}}

\newcommand{\B}{{\mathcal B}}
\newcommand{\M}{{\mathcal M}}
\newcommand{\snake}{\text{\Large $\mathfrak{s}$}}
\newcommand{\Snake}{\widehat{\text{\Large $\mathfrak{s}$}}}
\newcommand{\e}{\mathbf e}
\newcommand{\one}{\boldsymbol{1}}

\begin{document}

\date{}
\title{Two-colorings of normed spaces without \\ long monochromatic unit arithmetic progressions}
\author{
	Valeriya Kirova\thanks{MSU, Moscow, Russia. Email:~\href{mailto:kirova_vo@mail.ru}{\tt kirova\textunderscore vo@mail.ru}.}
	\and  
	Arsenii Sagdeev\thanks{MIPT, Moscow, Russia and Alfréd Rényi Institute of Mathematics, Budapest, Hungary. Supported in part by the \href{https://rscf.ru/project/21-71-10092/}{RSF grant N 21-71-10092} and by ERC Advanced Grant `GeoScape'. The author is also a winner of Young Russian Mathematics Contest and would like to thank its sponsors and jury. Email:~\href{mailto:sagdeevarsenii@gmail.com}{\tt sagdeevarsenii@gmail.com}.}
}

\maketitle

\begin{abstract}
	Given a natural $n$, we construct a two-coloring of $\mathbb{R}^n$ with the maximum metric satisfying the following. For any finite set of reals $S$ with diameter greater than $5^{n}$ such that the distance between any two consecutive points of $S$ does not exceed one, no isometric copy of $S$ is monochromatic. As a corollary, we prove that any normed space can be two-colored such that all sufficiently long unit arithmetic progressions contain points of both colors.
\end{abstract}

\section{Introduction}

For an $n$-dimensional normed space $\R_N^n$, its {\it chromatic number} $\chi(\R_N^n)$ is the smallest $r$ such that there exists a coloring of the points of $\R_N^n$ with $r$ colors, i.e. an {\it $r$-coloring}, and with no two points of the same color unit distance apart. A general result by Kupavskii \cite{Kup} establishes an upper bound on this quantity depending only on the dimension:
$\chi(\R_N^n) \le (4+o(1))^n$
as $n \to \infty$.

This notion was most extensively studied for the $n$-dimensional $\ell_p$-spaces\footnote{Recall that the $\ell_p$-norm of $\x \in \R^n$ is given by $\|\x\|_p \coloneqq \big(\sum_{i}|x_i|^p\big)^{1/p}$ for any real $p\ge 1$, and in case $p = \infty$ by $	\|\x\|_{\infty} \coloneqq \max_i |x_i|$.} $\mathbb{R}^n_{p}$ and especially for the Euclidean spaces $\R_2^n$. Currently, the best known bounds on the plane are $5 \le \chi(\R_2^2) \le 7$. The lower bound is a relatively recent breakthrough by de Grey \cite{degrey} (reproved quickly after by Exoo and Ismailescu \cite{Exoo}). The upper bound here is classical. See also Soifer’s account of the history of this problem in \cite{Soifer}. As for the growing dimension case, currently the best asymptotic lower and upper bounds belong to Raigorodskii \cite{Rai3} and Larman and Rogers \cite{Lar, Pros2020_LR} respectively: $(1.239+o(1))^n \le \chi(\R_2^n) \le (3+o(1))^n$ as $n \to \infty$. For non-Euclidean $\ell_p$-spaces the value $\chi(\R_p^n)$ was shown to grow exponentially with $n$ as well (for instance, see a paper \cite{Rai4} and two surveys \cite{RaiSur1, RaiSur2} by Raigorodskii). The case of the Chebyshev spaces $\R_\infty^n$ stands out here because of the folklore equality $\chi( \R^n_\infty) = 2^n$ that holds for all $n \in \mathbb{N}$.

To generalize these problems, one can forbid more complex configurations to be monochromatic. Given a normed space $\R_N^n$ and a subset $\M \subset \R^n$, the {\it chromatic number $\chi(\R_N^n,\M)$} is the smallest $r$ such that there exists an $r$-coloring of $\R^n$ with no monochromatic {\it $N$-isometric copy}\footnote{A subset $\M' \subset \R^n$ is called an {\it $N$-isometric copy} of $\M$ if there exists a bijection $f: \M \to \M'$ such that $\|\x-\y\|_N = \|f(\x)-f(\y)\|_N$ for all $\x,\y \in \M$.} of $\M$. In these terms, $\chi(\R_N^n) = \chi(\R_N^n,I)$, where $I$ is a two-point set.

A systematic study of this notion began with three classic papers by Erd\H{o}s, Graham, Montgomery, Rothschild, Spencer, and Straus \cite{EGMRSS1,EGMRSS2,EGMRSS3}, and now grew into a separate branch of combinatorics, see a survey \cite{Graham2017} by Graham. In the Euclidean case, the most extensively studied question here is the following. Given $\M$, determine if it is {\it Ramsey}, i.e., if the value $\chi(\R_2^n,\M)$ tends to infinity as $n$ grows. The sets of vertices of simplices, boxes \cite{FrRod} and regular polytopes \cite{Kriz1,Cant} are known to be Ramsey (for the explicit bounds on these chromatic numbers see \cite{Naslund, Pros2018_ExpRams, Sag2018_CartProd}). However, the problem of determining all Ramsey sets remains widely open in general, and even the conjectures on the answer spur debates, see \cite{Leader}.

One of the simplest examples of non-Ramsey sets are provided by one-dimensional configurations. Following the paper \cite{KupSag}, given a sequence of positive reals $\lambda_1,\dots, \lambda_k$, we call a set $\{0, \lambda_1, \lambda_1+\lambda_2, \dots,\sum_{t=1}^{k}\lambda_t\} \subset \R$ a {\it baton} and denote it by $\B(\lambda_1,\dots, \lambda_k)$. In case  $\lambda_1=\dots=\lambda_k=1$, i.e., if the set is just a unit arithmetic progression, we simply denote it by $\B_k$ for a shorthand. Erd\H{o}s et al. \cite{EGMRSS1} showed that any baton $\B$ of at least three points is not Ramsey, since $\chi(\R_2^n, \B) \le 16$ for all $n \in \N$. Moreover, they proved that
\begin{equation} \label{eq_Erd}
	\chi(\R_2^n, \B_k) = 2
\end{equation}
for all $k \ge 5$ and for all natural $n$. (Note that it is unknown \cite{Graham2017} whether the values $5$ and $16$ here are tight.) For asymmetric versions of these results see \cite{ArTsat,ArTsat2,ConFox}.

The goal of our paper is to find the analogues of \eqref{eq_Erd} for non-Euclidean normed spaces. More precisely, we positively resolve the following general problem for a wide class of normed spaces.

\begin{Problem} \label{prob}
	Is it true that for any normed space $\R_N^n$, there is $k=k(\R_N^n)$ such that $\chi(\R_N^n, \B_k) = 2$?
\end{Problem}

In a recent series of papers \cite{FKS,FKS2,KupSag_RMS,KupSag}, the chromatic numbers $\chi(\R_\infty^n, \M)$ of the $n$-dimensional Chebyshev spaces $\R_\infty^n$ were studied. In particular, it was proven in~\cite{KupSag} that
\begin{equation} \label{eq_KS}
	\chi(\R_\infty^n, \B_k) \ge \left(\frac{k+1}{k}\right)^n
\end{equation}
for all $k,n \in \N$. This inequality shows that, unlike the Euclidean case, for any given $k$, every two-coloring of $\R^n$ contains a monochromatic $\ell_\infty$-isometric copy of $\B_k$ whenever the dimension $n$ is large enough in terms of $k$. However, the next theorem, which is the main result of our paper, shows that this is not that case in the `opposite' setting, when $k$ is sufficiently large in terms of $n$.

\begin{Theorem} \label{th_main}
	Given $n \in \N$, there exists a two-coloring of $\R^n$ with no monochromatic $\ell_\infty$-isometric copies of all batons $\B(\lambda_1,\dots, \lambda_k)$ such that $\max_t\lambda_t \le 1$ and $\sum_{t=1}^k\lambda_t \ge 5^n$. In particular, for all $n \in \N$ and $k \ge 5^{n}$, we have $\chi(\R_\infty^n,\B_k)=2.$
\end{Theorem}

Our proof of this theorem is constructive. Note that in this paper we make no attempts to optimize the constant $5^n$ in the statement in order to keep the proof comprehensible, see further discussion in Section~\ref{sec_concl}. 

We apply Theorem~\ref{th_main} to get a positive solution of Problem~\ref{prob} for many normed spaces $\R_N^n$ other than $\R_\infty^n	$. The main obstacle on this way is the following. An arbitrary $N$-isometric copy of a baton $\B$ in $\R^n$ is not necessarily an $\ell_\infty$-isometric copy of some other baton\footnote{Indeed, the set of three points $(0,0)$, $(1,0)$, and $(1,1)$ on the plane is an $\ell_1$-isometric copy of $\B_2$. At the same time, all the $\ell_\infty$-distances between them are unit, and thus none of the distances equals the sum of two others.}. However, this obstacle vanishes once we consider only {\it collinear} $N$-isometric copies of $\B$. In the special case $\B=\B_k$, we call its collinear $N$-isometric copies {\it unit arithmetic progressions in $\R_N^n$ of length $k+1$}.

\begin{Corollary} \label{cor1}
	For any normed space $\R_N^n$ there exists a real $\delta=\delta(\R_N^n)$ such that the following holds. There exists a two-coloring of $\R^n$ with no monochromatic collinear $N$-isometric copies of all batons $\B(\lambda_1,\dots, \lambda_k)$ such that $\max_t\lambda_t \le 1$ and $\sum_{t=1}^k\lambda_t \ge \delta$. In particular, all sufficiently long unit arithmetic progressions in $\R_N^n$ contain points of both colors under this coloring.
\end{Corollary}

Observe that if the norm $N$ on $\R^n$ is {\it strictly convex}\footnote{The norm $N$ on $\R^n$ is called {\it strictly convex} if and only if for all $\x,\y \in \R^n,$ the equality $\|\x+\y\|_N = \|\x\|_N+\|\y\|_N$ implies that $\x$ and $\y$ are collinear.}, then each $N$-isometric copy of any baton $\B$ must be collinear. Thus, Corollary~\ref{cor1} yields a positive solution of Problem~\ref{prob} for such normed spaces. For instance, it is well-known that the $\ell_p$-norm is strictly convex for all $1<p<\infty$.

The next corollary deals with somewhat the `opposite' situation, when the `unit ball' of the norm is a centrally symmetric convex polytope (and thus such norm is clearly not strictly convex).

\begin{Corollary} \label{cor2}
	Let $\R_N^n$ be a normed space whose unit ball is a centrally symmetric convex polytope in $\R^n$ with $2f$ facets. Then there exists a two-coloring of $\R^n$ with no monochromatic $N$-isometric copies of all batons $\B(\lambda_1,\dots, \lambda_k)$ such that $\max_t\lambda_t \le 1$ and $\sum_{t=1}^k\lambda_t \ge 5^f$. In particular, for all $k \ge 5^{f}$, we have $\chi(\R_N^n,\B_k)=2.$
\end{Corollary}

Clearly, this result can be applied to the $n$-dimensional space $\R_1^n$ with the Manhattan distance, whose unit ball is a cross-polytope with $2^n$ facets. Along with Corollary~\ref{cor1}, these results solve Problem~\ref{prob} for all $\ell_p$-spaces. However, the general case remains open.

\vspace{5mm}

We organize the remainder of the paper as follows. Section~\ref{sec_prel} contains some preliminary technical statements: a classification of all $\ell_\infty$-isometric copies of batons in $\R^n$, a notion of a `snake hypersurface', and proofs of its basic properties. In Sections~\ref{sec_main_proof} we construct an explicit coloring of $\R^{n+1}$ based on these hypersurfaces and prove Theorem~\ref{th_main}. In Section~\ref{sec_cor} we deal with normed spaces other than $\R_\infty^n$ and prove Corollaries~\ref{cor1} and~\ref{cor2}. Finally, in Section~\ref{sec_concl} we make some further comments and state more open problems.

\section{Preliminaries} \label{sec_prel}

\subsection{Embeddings of the batons into $\R_\infty^n$} \label{sec_baton}

Let $\B = \B(\lambda_1,\dots,\lambda_k) \subset \R$ be an arbitrary fixed baton. It is clear that, for all $x \in \R$, both a {\it translation} $x+\B \coloneqq \{x, x+\lambda_1,\dots, x+\sum_{t=1}^k\lambda_t\}$ and a {\it reflection} $x-\B \coloneqq \{x, x-\lambda_1,\dots, x-\sum_{t=1}^k\lambda_t\}$ are isometric copies of $\B$ (regardless of the considered norm on $\R$). Moreover, it is not hard to see that there are no other isometric copies of $\B$ on the line.

The following simple lemma extends these ideas to the multidimensional Chebyshev spaces $\R_\infty^n$. Roughly speaking, it states that an arbitrary set of points of $\R^n$ forms an $\ell_\infty$-isometric copy of $\B$ if and only if one of the projections of the set on basic axes is either a translation or a reflection of $\B$, while all the other projections do not affect the distances between the points. This lemma appeared earlier in \cite{FKS}, but we give its full short proof below for completeness.

In the remainder of the text we will use the notation $\x = (x_1,\dots,x_n)$ for points $\x\in \mathbb{R}^n$.

\begin{Lemma} \label{lem_baton_embed}
	Let $k,n \in \N$ and $\lambda_1,\dots,\lambda_k$ be a sequence of positive reals. Then the sequence $\x^0,\ldots, \x^k$ of points in $\R^n$ is $\ell_\infty$-isometric to $\B = \B(\lambda_1,\dots,\lambda_k)$ (in that order) if and only if the following two conditions hold. First, there exists an $i \le n$ such that the sequence $x_i^0,\ldots,x_i^k$ is either a translation or a reflection of $\B$. Second, for any $j\le n$ and any $s \le k$ we have  $|x_j^{s}-x_j^{s-1}| \le |x_i^s-x_i^{s-1}| = \lambda_s.$
\end{Lemma}

\begin{proof}
	Let the sequence $\x^0,\ldots,\x^k$ of points in $\R^n$ be an $\ell_\infty$-isometric copy of $\B$. For a shorthand, put $\sigma_s \coloneqq \sum_{t=1}^{s} \lambda_t$ for all $s=0,\dots,k$. Let $i$ be a coordinate such that $|x^k_i-x^0_i| =\sigma_k$ (there must exist at least one such coordinate, since $\max_j|x^k_j-x^0_j|=\|\x^k-\x^0\|_\infty = \sigma_{k}$). Then it should be clear that, for all $s=1,\ldots, k-1$, we have a unique choice of $x_i^s$ so that two inequalities $|x^s_i-x^0_i| \le \|\x^s-\x^0\|_\infty = \sigma_s$ and $|x^k_i-x^s_i| \le \|\x^k-\x^s\|_\infty = \sigma_k-\sigma_s$ hold simultaneously. Moreover, both inequalities must hold with equality in that case. This concludes the proof of the first part. The second part of the statement is trivial, since for any $j \le n$ and any $s \le k$, we clearly have $|x_j^s-x_j^{s-1}| \le \|\x^s-\x^{s-1}\|_\infty = \lambda_s = |x_i^s-x_i^{s-1}|$.
	
	To prove the opposite direction, let us assume that the sequence $\x^0,\ldots,\x^k$ of points in $\R^n$ satisfies both conditions of the lemma. Given $0 \le l < r \le k$, observe that the first condition implies that $|x_i^r-x_i^l| = \sigma_r-\sigma_l$. Moreover, from the second condition and the triangle inequality it follows that
	$$|x_j^r-x_j^l| \le \sum_{s=l+1}^{r} |x_j^s-x_j^{s-1}| \le \sum_{s=l+1}^{r} |x_i^s-x_i^{s-1}| = \sum_{s=l+1}^{r}\lambda_s = \sigma_r-\sigma_l$$
	for all $j \le n$. Hence, we have $\|\x^r-\x^l\|_{\infty} = |x_i^r-x_i^l| = \sigma_r-\sigma_l$. So, the sequence $\x^0,\ldots, \x^k$ is indeed an $\ell_\infty$-isometric copy of $\B$.
\end{proof}

We say that a given $\ell_\infty$-isometric copy of $\B$ {\it has direction $i$} if its projection on the $i$-th basic axis forms either a translation or a reflection of $\B$. Lemma~\ref{lem_baton_embed} implies that each $\ell_\infty$-isometric copy of $\B$ has at least (but not necessarily exactly) one direction.

\subsection{Snake hypersurfaces} \label{sec_hyper}

In the present section we introduce a notion of a `snake hypersurface' and prove that it shares many basic properties with hyperplanes. Informally speaking, Lemma~\ref{lem_homot} shows that the family of snake hypersurfaces is closed under scaling. Lemma~\ref{lem_dist} ensures that the smallest distance between two parallel copies of a snake hypersurface is attained on any pair of the corresponding points. Lemma~\ref{lem_inj} states that an orthogonal projection on a snake hypersurface is well defined. Finally, Lemma~\ref{lem_contour} affirms that a hyperplane section of a snake hypersurfaces is also a snake hypersurface but of dimension one less. These properties will play a crucial role in constructing a family of `snake colorings' and proving the main result of this paper in Section~\ref{sec_main_proof}.

We begin with setting up some notation. Let $\e_1,\dots,\e_n$ be the standard basis vectors for $\R^n$, and $\one_n \in \R^n$ be their sum, i.e., a vector with all $n$ coordinates being unit\footnote{Note that here and in what follows we do not distinguish points from their position vectors.}. For all $m < n$, we also identify $\R^m$ with an $m$-dimensional subspace of $\R^n$ which is spanned by its first $m$ basis vectors $\e_1,\dots,\e_m$. For instance, we identify the points $(1,2)\in \R^2$ and $(1,2,0) \in \R^3$. For two subsets $A,B \subset \R^n$, we denote their {\it Minkowski sum} $\{\mathbf{a}+\mathbf{b}: \mathbf{a} \in A, \mathbf{b} \in B\} \subset \R^n$ by $A+B$ as usual. We call the Minkowski sum {\it injective} if for all $\mathbf{a},\mathbf{a}' \in A, \mathbf{b},\mathbf{b}' \in B,$ the equality $\mathbf{a}+\mathbf{b} = \mathbf{a}'+\mathbf{b}'$ implies that $\mathbf{a}=\mathbf{a}'$ and $\mathbf{b}=\mathbf{b}'$. Finally, for a subset $A \subset \R^n$ and for a set of reals $I \subset \R$, we denote the set of pairwise products $\{i\cdot\mathbf{a}: i \in I, \mathbf{a} \in A\} \subset \R^n$ by $I\cdot A$.

%\newpage

Given $n \in \N$, a {\it snake hypersurface} $\snake^n(a_1,b_1,\dots,a_n,b_n)$ is an $n$-dimensional piecewise linear hypersurface in $\R^{n+1}$ that depends on $2n$ positive real parameters. The definition is by induction on $n$. Let $\snake^0(\varnothing) \coloneqq \{0\}$ be just the origin of the line. For $n>0$, we define $\snake^n(a_1,b_1,\dots,a_n,b_n)$ by
\begin{align} \label{eq_snake_def}
	\snake^n(a_1,b_1,\dots,a_n,b_n) =&\ \snake^{n-1}(a_1,b_1,\dots,a_{n-1},b_{n-1}) \nonumber \\
	&+\Z\cdot\{a_n\cdot\e_{n+1}-b_n\cdot\one_n\}+[0,a_n)\cdot\e_{n+1}\cup(0,b_n]\cdot\one_n.
\end{align}
We visualize this definition in case $n=1,2$ in Figure~\ref{Fig1}.

\begin{figure}[h]
	\centering
	\includegraphics[scale=0.45]{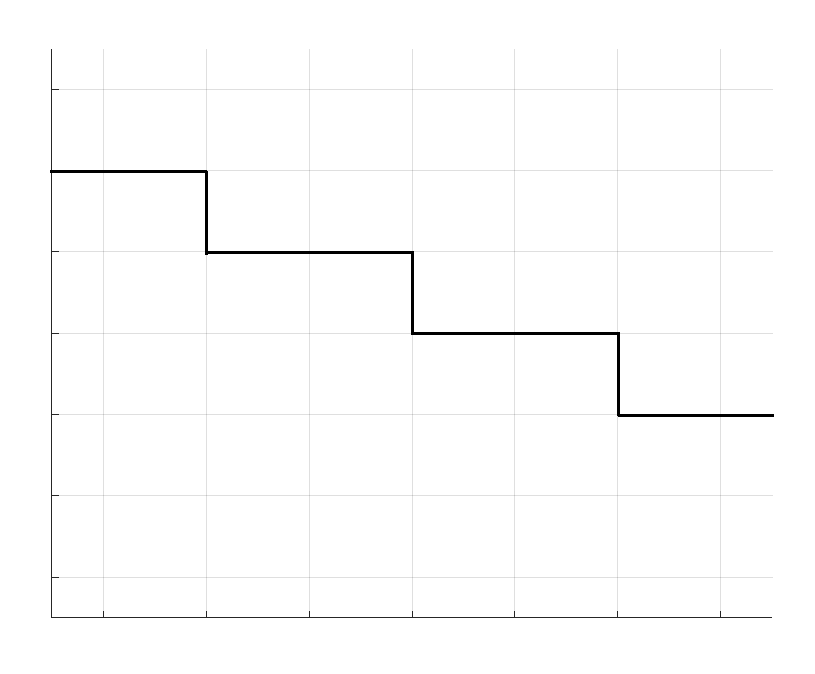}
	\includegraphics[scale=0.27]{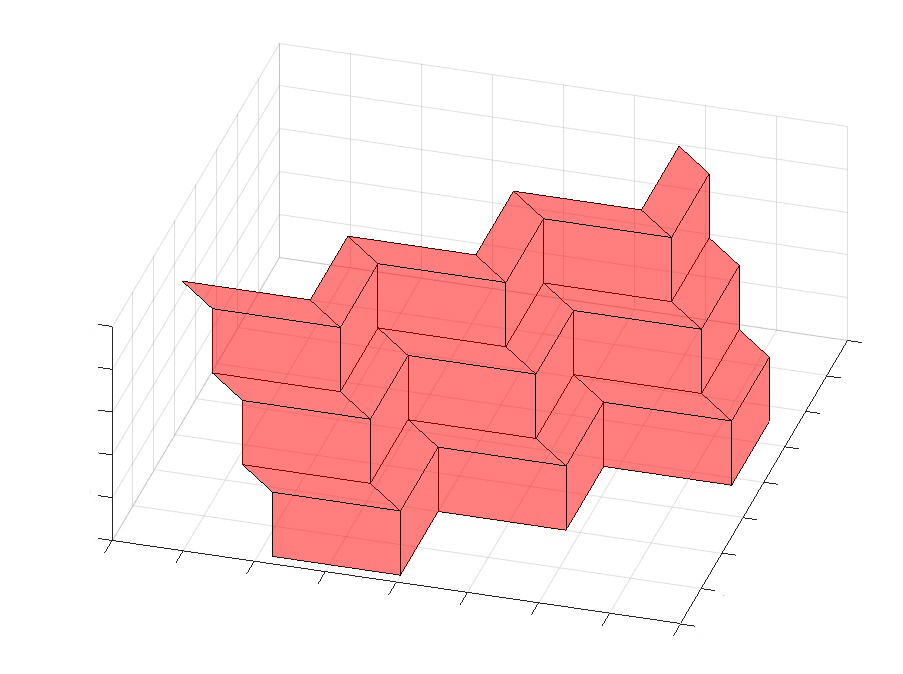}
	\caption{Snake hypersurfaces on the plane and in space}
	\label{Fig1}
\end{figure}

In what follows, whenever this does not cause confusion, we omit the parameters of a snake hypersurface and denote $\snake^n(a_1,b_1,\dots,a_n,b_n)$ simply by $\snake^n$ for a shorthand.

Now we give the aforementioned basic properties of these hypersurfaces as separate lemmas.

\begin{Lemma} \label{lem_homot}
	For all $n \in \N$, positive reals $a_1,b_1,\dots,a_n,b_n$ and $\mu$, we have
	\begin{equation*}
		\snake^n(\mu a_1,\mu b_1,\dots,\mu a_n,\mu b_n) = \mu \cdot \snake^n(a_1,b_1,\dots,a_n,b_n).
	\end{equation*}
\end{Lemma}
\begin{proof}
	The proof is by induction on $n$. In case $n=0$ we have $\{0\} = \mu \cdot \{0\},$ and there is simply nothing to prove. If $n>1$, we combine the induction hypothesis with \eqref{eq_snake_def} to conclude that
	\begin{align*}
		\snake^n(\mu a_1,\mu b_1,\dots,\mu a_n,\mu b_n) =&\ \snake^{n-1}(\mu a_1,\mu b_1,\dots,\mu a_{n-1},\mu b_{n-1})\\
		&+\Z\cdot\{\mu a_n\cdot\e_{n+1}-\mu b_n\cdot\one_n\}+[0,\mu a_n)\cdot\e_{n+1}\cup(0, \mu b_n]\cdot\one_n \\
		=&\ \mu\cdot \snake^{n-1}(a_1,b_1,\dots,a_{n-1},b_{n-1}) \\
		&+\mu\cdot\Z\cdot\{a_n\cdot\e_{n+1}-b_n\cdot\one_n\}+\mu\cdot[0,a_n)\cdot\e_{n+1}\cup\mu\cdot(0, b_n]\cdot\one_n \\
		=&\ \mu\cdot \Big( \snake^{n-1}(a_1,b_1,\dots,a_{n-1},b_{n-1}) \\
		&+\Z\cdot\{a_n\cdot\e_{n+1}-b_n\cdot\one_n\}+[0,a_n)\cdot\e_{n+1}\cup(0, b_n]\cdot\one_n \Big) \\
		=&\ \mu\cdot \snake^{n}(a_1,b_1,\dots,a_{n},b_{n}).
	\end{align*}
	This completes the proof. 
\end{proof}

\begin{Lemma} \label{lem_dist}
	For all $n \in \N$, positive reals $a_1,b_1,\dots,a_n,b_n$ and $t$ the following holds. For any two points $\x,\y \in \snake^n = \snake^n(a_1,b_1,\dots,a_n,b_n)$, we have $\|\x+t\cdot\one_{n+1}-\y\|_\infty \ge t$. In other words, the $\ell_\infty$-distance between $\snake^n$ and a translation $\snake^n+t\cdot\one_{n+1}$ equals $t$.
\end{Lemma}
\begin{proof}
	Let us use induction on $n$. If $n=0$, there is nothing to do, since $|0+t-0|=t$. So, we turn to the induction step. Fix $n \ge 1$, and assume that there exist $\x,\y \in \snake^n$ such that
	\begin{equation} \label{eq_lem_dist_1}
		\|\x+t\cdot\one_{n+1}-\y\|_\infty < t.
	\end{equation}
	We use the definition~\eqref{eq_snake_def} of the snake hypersurface to represent $\x$ as a sum
	\begin{equation*}
		\x=\x'+c_{\x}(a_n\cdot\e_{n+1}-b_n\cdot\one_{n})+v_{\x}\cdot\e_{n+1}+w_{\x}\cdot\one_n,
	\end{equation*}
	where $\x' \in \snake^{n-1}, c_{\x} \in \Z, v_{\x} \in [0,a_n), w_{\x} \in [0,b_n],$ and either $v_{\x}$ or $w_{\x}$ is equal to zero. Then we represent $\y$ similarly.
	
	Considering only the last coordinate, we deduce from~\eqref{eq_lem_dist_1} that
	\begin{equation} \label{eq_lem_dist_2}
		|a_n(c_{\x}-c_{\y})+v_{\x}-v_{\y}+t| < t.
	\end{equation}
	As far as $v_{\x}-v_{\y} > -a_n,$ the last inequality implies that $c_{\x} - c_{\y} \le 0$.
	
	Similarly, one can consider the first $n$ out of $n+1$ coordinates in~\eqref{eq_lem_dist_1} to get that
	\begin{equation*}
		\|\x'+\big( b_n(c_{\y}-c_{\x})+w_{\x}-w_{\y} +t \big)\cdot\one_{n}-\y'\|_{\infty} < t.
	\end{equation*}
	By the induction hypothesis, the last inequality implies that
	\begin{equation} \label{eq_lem_dist_3}
		|b_n(c_{\y}-c_{\x})+w_{\x}-w_{\y} +t|<t.
	\end{equation}
	Since $w_{\x}-w_{\y} \ge -b_n,$ it follows from~\eqref{eq_lem_dist_3} that $c_{\y}-c_{\x} \le 0$.
	
	Hence, if both inequalities~\eqref{eq_lem_dist_2} and~\eqref{eq_lem_dist_3} hold, then $c_{\x}=c_{\y}$. Under this assumption, these inequalities turn into	$|v_{\x}-v_{\y}+t| < t$  and $|w_{\x}-w_{\y} +t|<t$, respectively. In particular, we have $v_{\x}<v_{\y}$ and $w_{\x}<w_{\y}$. However, this is a contradiction, because either $v_{\y}$ or $w_{\y}$ is equal to zero. This establishes the induction step.
	
	So, we have shown that for all $\x,\y \in \snake^n$, we have $\|\x+t\cdot\one_{n+1}-\y\|_\infty \ge t$. Besides, if $\x=\y$, then the last inequality holds with the equality. This proves that the $\ell_\infty$-distance between the hypersurfaces $\snake^n$ and $\snake^n+t\cdot\one_{n+1}$ equals $t$.
\end{proof}

\begin{Lemma} \label{lem_inj}
	For all $n \in \N$ and positive reals $a_1,b_1,\dots,a_n,b_n$, the following two statements are valid. First, the Minkowski sum from the right-hand side of~\eqref{eq_snake_def} is injective. Second, we have
	\begin{equation*}
		\R^{n+1} = \snake^{n}(a_1,b_1,\dots,a_{n},b_{n})+ \R\cdot\one_{n+1},
	\end{equation*}
	ans this sum is also injective.
\end{Lemma}
\begin{proof}
	As in the case of previous lemmas, the proof here is by induction on $n$. The base case $n=0$ is immediate. Indeed, the first statement is degenerate in this case, while the second one is trivial, since the sum $\R = \{0\}+\R\cdot\one_1$ is clearly injective. So, we turn to the induction step.
	
	Assume that $n\ge 1$ and consider the Minkowski sum
	\begin{equation} \label{eq_lem_inj_1}
		\Sigma \coloneqq \snake^{n-1} +\Z\cdot\{a_n\cdot\e_{n+1} -b_n\cdot\one_n\} +[0,a_n)\cdot\e_{n+1}\cup(0,b_n]\cdot\one_n + \R\cdot\one_{n+1}.
	\end{equation}
	
	Observe that the last term of the right-hand side of~\eqref{eq_lem_inj_1} allows us to subtract an arbitrary vector collinear to $\one_{n+1}$ from any point of the preceding terms of the Minkowski sum. Moreover, this operation does not change whether the sum is injective or not. In particular, for all $v \in [0,a_n)$, we can replace $v\cdot\e_{n+1}$ by $v\cdot\e_{n+1}-v\cdot\one_{n+1} = -v\cdot\one_{n}$. So, we replace $[0,a_n)\cdot\e_{n+1}$ from the penultimate term of~\eqref{eq_lem_inj_1} by $(-a_n,0]\cdot\one_{n}$ and conclude that
	\begin{equation*} \label{eq_lem_inj_2}
		\Sigma =\ \snake^{n-1} +\Z\cdot\{a_n\cdot\e_{n+1}-b_n\cdot\one_n\}+(-a_n,b_n]\cdot\one_{n} + \R\cdot\one_{n+1}.
	\end{equation*}
	
	Similarly, we replace $\Z\cdot\{a_n\cdot\e_{n+1}-b_n\cdot\one_n\}$ by $\Z\cdot\{-(a_n+b_n)\cdot\one_{n}\}$ in the right-hand side of the last equality. Then, observe that
	\begin{equation*}
		\Z\cdot\{-(a_n+b_n)\cdot\one_{n}\} + (-a_n,b_n]\cdot\one_{n} = \R\cdot\one_{n}
	\end{equation*}
	and this sum is injective. Therefore,
	\begin{equation*}
		\Sigma =\ \snake^{n-1} +\R\cdot\one_{n} + \R\cdot\one_{n+1} = \R^n+ \R\cdot\one_{n+1},
	\end{equation*}
	where the last equality hold by induction. Finally, it is easy to check that $\R^n+ \R\cdot\one_{n+1} = \R^{n+1}$ and this sum is injective.
	
	Summing up, we have seen that the Minkowski sum~\eqref{eq_lem_inj_1} is equal to $\R^{n+1}$ and injective. Taking into account the definition~\eqref{eq_snake_def} of the snake hypersurface, this observations yields both the desired statements and completes the proof. 
\end{proof}

Lemma~\ref{lem_inj} implies that
\begin{equation} \label{eq_Snake_0}
	\R^{n+1} = \snake^n + \R\cdot\one_{n+1} = \snake^n + [0,1)\cdot\one_{n+1}+\Z\cdot\one_{n+1} = \Snake^n + \Z\cdot\one_{n+1},
\end{equation}
where
\begin{equation} \label{eq_Snake_1}
	\Snake^n = \Snake^n(a_1,b_1,\dots,a_{n},b_{n}) = \snake^n(a_1,b_1,\dots,a_{n},b_{n}) + [0,1)\cdot\one_{n+1},
\end{equation}
and all these sums are injective.

The following technical lemma describes the `contours' of $\Snake^n$, i.e., its intersections
\begin{equation} \label{eq_Snake_cont}
	\Snake^n(h) \coloneqq \big\{(x_1,\dots,x_n) \in \R^n : (x_1,\dots,x_n,h) \in \Snake^n\big\}
\end{equation}
with the hyperplanes $x_{n+1} = h$ in case $a_n>1$.

\begin{Lemma} \label{lem_contour}
	For all $n \in \N$, positive reals $a_1,b_1,\dots,a_n,b_n$ such that $a_n>1$, and for all $m \in \Z$, the following holds. First, if $h \in [ma_n,ma_n+1)$, then
	\begin{equation} \label{eq_lem_contour_s1}
		\Snake^n(h) = \snake^{n-1}-mb_n\cdot\one_{n}+[0,b_n+1)\cdot\one_{n}.
	\end{equation}
	Second, if $h \in [ma_n+1,ma_n+a_n)$, then
	\begin{equation} \label{eq_lem_contour_s2}
		\Snake^n(h) = \snake^{n-1}-mb_n\cdot\one_{n}+[0,1)\cdot\one_{n}.
	\end{equation}
\end{Lemma}
\begin{proof}
	We begin with the proof of the second statement.
	
	Fix $m \in \Z$ and $h \in [ma_n+1,ma_n+a_n)$. Recall that it follows from~\eqref{eq_snake_def} and~\eqref{eq_Snake_1} that each point $\x \in \Snake^n$ has a unique representation as a sum
	\begin{equation} \label{eq_lem_contour_1}
		\x=\x'+c_{\x}(a_n\cdot\e_{n+1}-b_n\cdot\one_{n})+v_{\x}\cdot\e_{n+1}+w_{\x}\cdot\one_n +u_{\x}\cdot\one_{n+1},
	\end{equation}
	where $\x' \in \snake^{n-1}, c_{\x} \in \Z, u_{\x} \in [0,1), v_{\x} \in [0,a_n), w_{\x} \in [0,b_n],$ and either $v_{\x}$ or $w_{\x}$ is equal to zero.
	
	The last coordinate of such $\x$ is equal to
	\begin{equation} \label{eq_lem_contour_21}
		x_{n+1}=c_{\x}a_n+v_{\x}+u_{\x},
	\end{equation}
	while the fist $n$ its coordinates are
	\begin{equation} \label{eq_lem_contour_22}
		(x_1,\dots,x_n)=\x'-c_{\x}b_n\cdot\one_{n}+(w_{\x}+u_{\x})\cdot\one_{n}.
	\end{equation}

	Observe that if $c_{\x}>m$, then we have $x_{n+1}\ge(m+1)a_n>h$. Similarly, if $c_{\x}<m$, then we have $x_{n+1}<(m-1)a_n+a_n+1\le h$. Thus, $x_{n+1}=h$ if and only if $c_{\x}=m, u_{\x} \in [0,1)$, $v_{\x} = h-ma_n-u_{\x} > 0$. Clearly, the last inequality implies that $w_{\x} = 0$. Now it follows from~\eqref{eq_lem_contour_22} that the first $n$ coordinates of such points form the set from the right-hand side of~\eqref{eq_lem_contour_s2}.
	
	Let us move on the fist part of the lemma. Fix $m \in \Z$ and $h \in [ma_n,ma_n+1)$. We use the same representation~\eqref{eq_lem_contour_1} of an arbitrary $\x \in \Snake^n$ as before. Again, if $c_{\x}>m$, then \eqref{eq_lem_contour_21} implies that $x_{n+1}\ge(m+1)a_n>h$. Similarly, if $c_{\x}<m-1$, then we have $x_{n+1}<(m-2)a_n+a_n+1< h$. However, there are two options for $c_{\x}$ this time: if $\x \in \Snake^n(h)$, then either $c_{\x}=m-1,$ or $c_{\x}=m$.
	
	In the former case, $x_{n+1}=h$ if and only if $v_{\x}+u_{\x} = h-(m-1)a_n$. Observe that if $u_{\x}\le h-ma_n$, then $v_{\x} = h-(m-1)a_n - u_{\x} \ge a_n$, which is impossible. At the same time, if $h-ma_n < u_{\x} < 1$, then we have $0<v_{\x}<a_n$ as required. Moreover, $w_{\x}=0$ for all such $\x$ since $v_{\x}$ is positive. Thus, it follows from~\eqref{eq_lem_contour_22} that the first $n$ coordinates of the points $\x \in \Snake^n$ such that $x_{n+1}=h$ and $c_{\x}=m-1$ form the set $S_1$ defined by
	\begin{equation*} \label{eq_lem_contour_31}
		S_1 = \snake^{n-1}-(m-1)b_n\cdot\one_{n}+(h-ma_n,1)\cdot\one_{n} = \snake^{n-1}-mb_n\cdot\one_{n} + I_1\cdot\one_{n}, 
	\end{equation*}
	where
	\begin{equation*}
		I_1 = (h-ma_n+b_n, b_n+1).
	\end{equation*}
	
	Now let us consider the other case, namely, $c_{\x}=m$. By~\eqref{eq_lem_contour_21}, we have $x_{n+1}=h$ if and only if $v_{\x}+u_{\x} = h-ma_n$. As far as $0\le h-ma_n<1$, there are two options to satisfy the last equation. The first one is to put $u_{\x} = h-ma_n, v_{\x}=0$, and thus the value of $w_{\x} \in [0,b_n]$ may be arbitrary. It follows from~\eqref{eq_lem_contour_22} that the first $n$ coordinates of such points form the set
	\begin{equation*}
		S_2 = \snake^{n-1}-mb_n\cdot\one_{n} + I_2\cdot\one_{n}, \mbox{ where } I_2 = [h-ma_n,h-ma_n+b_n].
	\end{equation*}
	The second option\footnote{Note that this option may be degenerate if $h=ma_n$.} is to take an arbitrary $u_{\x} \in [0,h-ma_n)$ and to put $v_{\x} = h-ma_n-u_{\x} > 0$. We shall have $w_{\x}=0$ for all such points since $v_{\x}$ is positive. Hence, by~\eqref{eq_lem_contour_22}, the first $n$ coordinates of these points form the set
	\begin{equation*}
		S_3 = \snake^{n-1}-mb_n\cdot\one_{n} + I_3\cdot\one_{n}, \mbox{ where } I_3 = [0,h-ma_n).
	\end{equation*}

	After considering all these possibilities, we conclude that
	\begin{equation*}
		\Snake^n(h) = S_1\cup S_2 \cup S_3 = \snake^{n-1}-mb_n\cdot\one_{n} + (I_1\cup I_2 \cup I_3)\cdot\one_{n}.
	\end{equation*}
	Finally, it is easy that $I_3\cup I_2 \cup I_1 = [0,b_n+1)$. This yields~\eqref{eq_lem_contour_s1} and finishes the proof.
\end{proof}

\begin{Corollary} \label{cor}
	In the notation of Lemma~\ref{lem_contour}, for all $m \in \Z$, $h \in [ma_n-a_n+1,ma_n+a_n)$, we have
	\begin{equation*}
		\Snake^n(h) \subset \Snake^n(ma_n) = \snake^{n-1}-mb_n\cdot\one_{n}+[0,b_n+1)\cdot\one_{n}.
	\end{equation*} 
\end{Corollary}

\section{Proof of Theorem~\ref{th_main}} \label{sec_main_proof}

In the present section we construct a special two-coloring of the space, a {\it snake coloring}, and prove that it has the desired property. The construction is based on the notion of a snake hypersurface, so we use the notation from the previous section.

Given $n \in \N$, put, with a foresight,
\begin{equation} \label{eq_th_main_param}
	a_i = a_i(n) \coloneqq \frac{7}{4}\cdot(5^n-5^{n-i}), \ b_i = b_i(n) \coloneqq 4\cdot5^{n-i}, \ i=0,\dots,n.
\end{equation}
For all $m \le n$, let us denote the snake hypersurface $\snake^m(a_1,b_1,\dots,a_m,b_m)$ by $\snake^m_n$ for a shorthand. According to~\eqref{eq_Snake_1}, put $\Snake^n \coloneqq \snake^n_n+[0,1)\cdot\one_{n+1}$. Recall that by~\eqref{eq_Snake_0}, we have 
$\R^{n+1} = \Snake^n + \Z\cdot\one_{n+1}$ and this sum is injective. Therefore, $\R^{n+1}$ can be represented as a disjoint union
\begin{equation*}
	\R^{n+1} = \mathfrak{R} \sqcup \mathfrak{B},
\end{equation*}
where
\begin{equation*}
	\mathfrak{R} \coloneqq \Snake^n + \{2z: z \in \Z\}\cdot\one_{n+1}, \ \ \mathfrak{B} \coloneqq \Snake^n + \{2z+1: z \in \Z\}\cdot\one_{n+1}.
\end{equation*}
Finally, let us color all points of $\mathfrak{R}$ and $\mathfrak{B}$ red and blue, respectively. We illustrate this coloring in case $n=1$ in Figure~\ref{Fig2}.

\begin{figure}[h]
	\centering
	\includegraphics[scale=0.5]{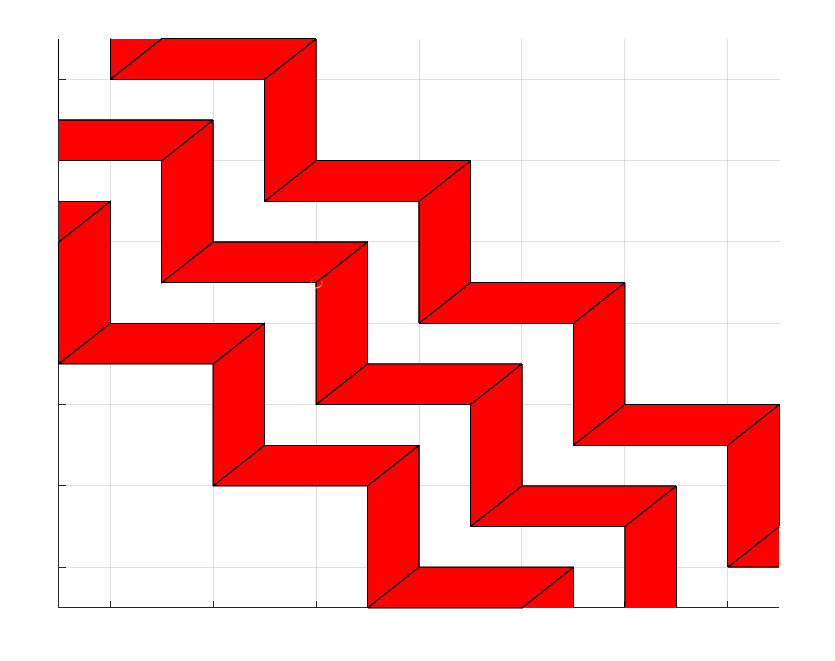}
	\caption{Part of the snake coloring of the plane}
	\label{Fig2}
\end{figure}

In order to prove Theorem~\ref{th_main}, we shall show that there are no monochromatic $\ell_\infty$-isometric copies of all batons $\B(\lambda_1,\dots, \lambda_k)$ such that $\max_t\lambda_t \le 1$ and $\sum_{t=1}^k\lambda_t \ge 5^{n+1}$. In view of the symmetry between $\mathfrak{R}$ and $\mathfrak{B}$, it is sufficient to prove the absence of only such {\it red} copies. Besides, from Lemma~\ref{lem_dist} it follows that for all $z,z' \in \Z$, $\x \in \Snake^n+2z\cdot\one_{n+1}$, $\y \in \Snake^n+2z'\cdot\one_{n+1}$, we have $\|\x-\y\|_\infty > 2\cdot|z-z'|-1$. In particular, the $\ell_\infty$-distance between any two points of different translates of $\Snake^n$ in $\mathfrak{R}$ is strictly greater than one. So, since $\max_t\lambda_t \le 1$, any red $\ell_\infty$-isometric copy of $\B(\lambda_1,\dots, \lambda_k)$ must lie entirely within some translate $\Snake^n+2z\cdot\one_{n+1}$ of $\Snake^n$ in $\mathfrak{R}$. Moreover, since $5^{n+1} > \frac{7}{4}\cdot5^n-\frac{3}{4} = a_n+1$, the following proposition finishes the proof of Theorem~\ref{th_main}.

\begin{Proposition} \label{prop_main}
	In the notation of this section, given $n \in \N$ and positive reals $\lambda_1,\dots,\lambda_k$, if $\max_t\lambda_t \le 1$ and $\sum_{t=1}^k\lambda_t \ge a_n+1$, then $\Snake^n$ contains no $\ell_\infty$-isometric copy of $\B = \B(\lambda_1,\dots, \lambda_k)$.
\end{Proposition}

The proof of this proposition is trivial if $n=0$. Indeed, in this case, $\Snake^0$ is just a half-opened interval $[0,1) \subset \R$ which clearly does not contain an isometric copy of $\B$ because $\mbox{diam}(\B)= \sum_{t=1}^k\lambda_t \ge a_0+1 = 1$.

Let us also illustrate the proof in case $n=1$ with the aid of Figure~\ref{Fig2}. It is easy to see that $\Snake^1$, which looks like a single red `snake', does not contain a {\it vertical} isometric copy of $\B$ since $\sum_{t=1}^k\lambda_t \ge a_1+1$. Indeed, the condition $\max_t\lambda_t \le 1$ makes the potential isometric copy too narrow to `jump' between different half-open vertical segments of the snake, and thus it has to fit within a single one of them of length $a_1+1$. By the same reason, $\Snake^1$ does not contain a {\it horizontal} isometric copy of $\B$ as well because $\sum_{t=1}^k\lambda_t \ge b_1+1$ (note that  $a_1=7 > 4=b_1$ by definition). However, it would be better to say, with a foresight, that $\Snake^1$ does not contain such isometric copies of $\B$ `by the induction hypothesis' since the case $n=0$ yields this upper bound on the maximal length of a baton inside a horizontal segment $[0,b_1+1)$. This concludes the proof of Proposition~\ref{prop_main} in case $n=1$ as far as there are only these two possible {\it directions} for an isometric copy of $\B$ (see Section~\ref{sec_baton}).

In the rest of this section we extend the argument from the previous paragraph and prove Proposition~\ref{prop_main} by induction on $n$ using only the trivial case $n=0$ as the base of induction\footnote{As a byproduct, this would provide a more formal proof for the case $n=1$ considered above.}.

So, let us assume that $n \ge 1$. Our argument consists of two steps, which are a natural generalization of the above dichotomy between vertical and horizontal isometric copies. The first step is to prove that no $\ell_\infty$-isometric copy of $\B$ in $\Snake^n$ has direction $n+1$, while during the second one we show that directions $1,\dots,n$ are also excluded. This would immediately finish the proof of Proposition~\ref{prop_main} (and thus the one of Theorem~\ref{th_main} as well), since Lemma~\ref{lem_baton_embed} ensures that each $\ell_\infty$-isometric copy of $\B$ in $\Snake^n \subset \R^{n+1}$ has at least one direction among the first $n+1$ positive integers.

\vspace{3mm}
{\noindent \bf Step 1:} {\it no $\ell_\infty$-isometric copy of $\B$ in $\Snake^n$ has direction $n+1$.}
\vspace{3mm}

Assume the contrary. Let the sequence $\x^0,\ldots,\x^k$ of points in $\Snake^n$ be an $\ell_\infty$-isometric copy of $\B$ that has direction $n+1$. Thus, their last coordinates $x^0_{n+1}, \dots, x^k_{n+1}$ form (in that order) either a translation, or a reflection of $\B$. Without loss of generality, assume that the former holds, i.e., that we have $x^s_{n+1}-x^{s-1}_{n+1}=\lambda_s \le 1$ for all $s = 1,\dots,k.$

Choose $m \in \N$ such that $ma_n \le x^0_{n+1}<(m+1)a_n$. A priori, there are two possibilities: either $x^k_{n+1} < (m+1)a_n+1$, or $x^k_{n+1} \ge (m+1)a_n+1$.

The former option leads to a contradiction since
\begin{equation*}
	\sum_{t=1}^k\lambda_t = x^k_{n+1}-x^0_{n+1} < (m+1)a_n+1-ma_n=a_n+1.
\end{equation*}

So, let us assume that $x^k_{n+1} \ge (m+1)a_n+1$. In this case, put
\begin{equation*}
	l \coloneqq \max\{s:x^s_{n+1}<(m+1)a_n\} \ \mbox{ and } \ r \coloneqq \min\{s:x^s_{n+1}\ge(m+1)a_n+1\}.
\end{equation*}
Now we estimate the distance $\|\x^r-\x^l\|_\infty$ in two ways to get a contradiction again.

On the one hand, it is clear that
\begin{equation*}
	(m+1)a_n \le x^{l+1}_{n+1} = x^l_{n+1}+\lambda_{l+1} \le x^l_{n+1}+1 < (m+1)a_n+1.
\end{equation*}
Similarly, we have
\begin{equation*}
	(m+1)a_n \le x^{r}_{n+1} - 1 \le x^{r}_{n+1}-\lambda_r = x^{r-1}_{n+1} < (m+1)a_n+1.
\end{equation*}
Thus,
\begin{align} \label{eq_step1_1}
	\|\x^r-\x^l\|_\infty =&\ x^r_{n+1}-x^l_{n+1} = (x^r_{n+1}-x^{r-1}_{n+1})+ (x^{r-1}_{n+1}-x^{l+1}_{n+1})+ (x^{l+1}_{n+1}-x^l_{n+1}) \nonumber \\
	=&\ \lambda_r+(x^{r-1}_{n+1}-x^{l+1}_{n+1})+\lambda_{l+1} <1+1+1=3.
\end{align}

On the other hand, it is clear that $\|\x^r-\x^l\|_\infty \ge \|\y^r-\y^l\|_\infty$, where $\y^s = (x^s_1,\dots,x^s_n)$ is a set of the first $n$ coordinates of $\x^s$, $s = 0,\dots,k$. Besides, we have $\y^l \in \Snake^n(x^l_{n+1})$ by definition, see~\eqref{eq_Snake_cont}. Recall that
\begin{equation*}
	(m+1)a_n > x^l_{n+1} = x^{l+1}_{n+1}-\lambda_{l+1} \ge (m+1)a_n-1\ge ma_n+1,
\end{equation*}
where the last inequality holds because $a_n \ge 2$, see~\eqref{eq_th_main_param}. Therefore, the second part of Lemma~\ref{lem_contour} implies that
\begin{equation*}
	\Snake^n(x^l_{n+1}) = \snake^{n-1}_n-mb_n\cdot\one_{n}+[0,1)\cdot\one_{n}.
\end{equation*}
Similarly, we have 
\begin{equation*}
	(m+1)a_n+1 \le x^r_{n+1} = x^{r-1}_{n+1}+\lambda_r < (m+1)a_n+1+1 \le (m+2)a_n.
\end{equation*}
Thus, by the second part of Lemma~\ref{lem_contour}, we have
\begin{equation*}
	\y^r \in \Snake^n(x^r_{n+1}) = \snake^{n-1}_n-(m+1)b_n\cdot\one_{n}+[0,1)\cdot\one_{n}.
\end{equation*}
Finally, let us apply Lemma~\ref{lem_dist} and conclude that the $\ell_\infty$-distance between $\Snake^n(x^l_{n+1})$ and $\Snake^n(x^r_{n+1})$ is equal to $b_n-1 = 3$, see~\eqref{eq_th_main_param}. In particular,  we have
\begin{equation*}
	\|\x^r-\x^l\|_\infty \ge \|\y^r-\y^l\|_\infty > 3.
\end{equation*}
However, the last inequality contradicts~\eqref{eq_step1_1}. This observation completes Step~1.

\vspace{3mm}
{\noindent \bf Step 2:} {\it for all $1\le i \le n$, no $\ell_\infty$-isometric copy of $\B$ in $\Snake^n$ has direction $i$.}
\vspace{3mm}

This proof shares similarities with the argument from the previous step. As earlier, assume the contrary. Let the sequence $\x^0,\ldots,\x^k$ of points in $\Snake^n$ be an $\ell_\infty$-isometric copy of $\B$ that has direction $i$ for some $1 \le i \le n$. Since their last coordinates do not determine the distances between these points (see Lemma~\ref{lem_baton_embed}), it is easy to see that the sequence $\y^0,\ldots,\y^k$ is also an $\ell_\infty$-isometric copy of $\B$ that has direction $i$, where $\y^s = (x^s_1,\dots,x^s_n)$, $s=0,\dots,k$.

Pick an index $l$ such that $x^l_{n+1} = \min_s x^s_{n+1}$. If this choice is not uniquely determined, we pick an arbitrary of them. Next, let us choose $m \in \N$ such that $(m-1)a_n+1 \le x^l_{n+1}<ma_n+1$. A priori, there are two possibilities: either $\max_s x^s_{n+1} < (m+1)a_n$, or $\max_s x^s_{n+1} \ge (m+1)a_n$.

In the former case, given $s=0,\dots,k$, we have $(m-1)a_n+1 \le x^s_{n+1} < (m+1)a_n$. Thus, Corollary~\ref{cor} implies that
\begin{equation} \label{eq_step2}
	\y^s \in \Snake^n(x^s_{n+1}) \subset \Snake^n(ma_n) = \snake^{n-1}_n-mb_n\cdot\one_{n}+[0,b_n+1)\cdot\one_{n}.
\end{equation}
This proves the following statement.

\begin{Claim} \label{cl1}
	The set $\{\y^0,\ldots,\y^k\}+mb_n\cdot\one_{n}$ is a subset of $\snake^{n-1}_n+[0,b_n+1)\cdot\one_{n}$ and forms an $\ell_\infty$-isometric copy of $\B$.
\end{Claim}

On the other hand, recall that $b_n=b_n(n)=4$ due to~\eqref{eq_th_main_param}. Besides, we have $\frac{a_i(n)}{5} = a_i(n-1)$ and $\frac{b_i(n)}{5} = b_i(n-1)$ for all $i=1,\dots,n-1$. Hence, it follows from Lemma~\ref{lem_homot} that
\begin{align*}
	\snake^{n-1}_n+[0,b_n+1)\cdot\one_{n} =&\ \snake^{n-1}\big(a_1(n),b_1(n),\dots,a_{n-1}(n),b_{n-1}(n)\big)+[0,5)\cdot\one_{n} \\
	=&\ 5\cdot\bigg(\snake^{n-1}\Big(\frac{a_1(n)}{5},\frac{b_1(n)}{5}, \dots,\frac{a_{n-1}(n)}{5},\frac{b_{n-1}(n)}{5}\Big)+[0,1)\cdot\one_{n}\bigg) \\
	=&\ 5\cdot\Big(\snake^{n-1}\big(a_1(n-1),b_1(n-1),\dots, a_{n-1}(n-1),b_{n-1}(n-1)\big)+[0,1)\cdot\one_{n}\Big) \\
	=&\ 5\cdot\big(\snake^{n-1}_{n-1}+[0,1)\cdot\one_{n}\big) = 5\cdot\Snake^{n-1}.
\end{align*}
Now we apply the induction hypothesis of Proposition~\ref{prop_main} to $\Snake^{n-1}$ and conclude the following.

\begin{Claim} \label{cl2}
	For all positive reals $\lambda_1',\dots,\lambda_{k'}'$, if $\max_t\lambda_t' \le 5$ and $\sum_{t=1}^{k'}\lambda_t' \ge 5a_{n-1}(n-1)+5$, then $\snake^{n-1}_n+[0,b_n+1)$ contains no $\ell_\infty$-isometric copy of $\B(\lambda_1',\dots,\lambda_{k'}')$.
\end{Claim}

In particular, the last statement implies that $\snake^{n-1}_n+[0,b_n+1)$ contains no $\ell_\infty$-isometric copy of $\B = \B(\lambda_1,\dots,\lambda_{k})$ because
\begin{equation*}
	\sum_{t=1}^k\lambda_t \ge a_n(n)+1 = \frac{7}{4}\cdot5^{n}-\frac{3}{4}> \frac{7}{4}\cdot5^{n}-\frac{15}{4} = 5a_{n-1}(n-1)+5.
\end{equation*}
However, this observation contradicts Claim~\ref{cl1}.

So, it remains only to consider the case when $x^s_{n+1} \ge (m+1)a_n$ for some $s$. Let $r$ be the closest to $l$ index such that $x^{r+1}_{n+1} \ge (m+1)a_n$. Without loss of generality, assume that $r>l$. It is clear the set $\{\x_s: l \le s \le r\}$ is an $\ell_\infty$-isometric copy of $\B' \coloneqq \B(\lambda_l,\dots,\lambda_{r-1})$. Thus, so is the set $\{\y_s: l \le s \le r\}$. Moreover, by the choice of $r$, we have $x^s_{n+1} < (m+1)a_n$ for all $l \le s \le r$. As earlier, these inequalities together with Corollary~\ref{cor} imply~\eqref{eq_step2} and prove the following statement.

\begin{Claim} \label{cl3}
	The set $\{\y_s: l \le s \le r\}+mb_n\cdot\one_{n}$ is a subset of $\snake^{n-1}_n+[0,b_n+1)\cdot\one_{n}$ and forms an $\ell_\infty$-isometric copy of $\B' = \B(\lambda_l,\dots,\lambda_{r-1})$.
\end{Claim}

Finally, we estimate the diameter of $\B'$ to get a contradiction with Claim~\ref{cl2}. Indeed,
\begin{align*}
	\sum_{t=l}^{r-1}\lambda_t =&\ \sum_{t=l}^{r}\lambda_t - \lambda_{r} =  \|\x^{r+1}-\x^l\|_\infty - \lambda_{r} \ge x^{r+1}_{n+1}-x^l_{n+1}-1 \\
	>&\ (m+1)a_n(n) - \big(ma_n(n)+1\big) - 1 = a_n(n)-2 = 5a_{n-1}(n-1)+5,
\end{align*}
where the final equality is via a straightforward calculation. This observation completes the proofs of both Proposition~\ref{prop_main} and Theorem~\ref{th_main}.

\section{Corollaries} \label{sec_cor}

\subsection{Proof of Corollary~\ref{cor1}}

Given $n \in \N$, let $N$ be a norm on $\R^n$. A well-known result states that any two norms on a finite-dimensional space are {\it equivalent}, i.e., each one is bounded by some linear function of another. In particular, there are positive reals $c$ and $C$ such that $c\|\x\|_N \le \|\x\|_\infty \le C\|\x\|_N$ for all $\x \in \R^n$. After a proper scaling, we may assume without loss of generality that $C=1$. We prove that the two-coloring of $\R^n$ from Theorem~\ref{th_main} has the desired property with $\delta=\delta(\R_N^n)\coloneqq 5^n/c$.

Indeed, let $\B=\B(\lambda_1,\dots, \lambda_k)$ be a baton such that $\max_t\lambda_t \le 1$ and $\sum_{t=1}^k\lambda_t \ge \delta$. Fix a sequence of collinear points $\x^0, \dots , \x^k$ in $\R^n$ that forms an $N$-isometric copy of $\B$. Consider a ratio
\begin{equation*}
	\mu \coloneqq \frac{\|\x^r-\x^l\|_\infty}{\|\x^r-\x^l\|_N}.
\end{equation*}
In is not hard to see that the collinearity of these points implies that the value of $\mu$ does not depend on the indexes $l$ and $r$ such that $0\le l < r \le k$. In particular, we conclude that the sequence $\x^0, \dots , \x^k$ forms an $\ell_\infty$-isometric copy of the baton $\B' = \B(\lambda_1',\dots, \lambda_k')$, where 
$\lambda_s'=\mu \lambda_s$ for all $s=0,\dots,k$. Moreover, we have $c \le \mu \le 1$ by construction.

Finally, observe that $\max_t\lambda_t' = \mu \cdot \max_t\lambda_t \le 1$ and $\sum_{t=1}^k\lambda_t' = \mu \cdot \sum_{t=1}^k\lambda_t \ge c \cdot \delta = 5^n$. Hence, Theorem~\ref{th_main} implies that the sequence $\x^0, \dots , \x^k$ contains points of both colors.

\subsection{Proof of Corollary~\ref{cor2}}

Let $\R_N^n$ be an $n$-dimensional normed space whose unit ball $U$ is a centrally symmetric convex polytope in $\R^n$ with $2f$ facets. The following proof is based on the well-known fact that $U$ can be considered as an intersection of an $f$-dimensional hypercube with some $n$-dimensional hyperplane that contains the origin (see \cite{Grun}, Theorem~5.1.3 or \cite{Klee}, Proposition~4.5). However, we write down all the details below for clarity.

Given $1\le i \le f$, let $\mathbf{c}_i \in \R^n$ be a vector orthogonal to the $i$-th pair of the opposite facets of $U$ such that their hyperplains are defined by the equations $\langle\mathbf{c}_i, \x \rangle =1$ and $\langle \mathbf{c}_i, \x \rangle =-1$, where $\langle\cdot,\cdot\rangle$ stands for the standard Euclidean dot product on $\R^n$. Then it is not hard to see that
\begin{equation*}
	U = \{\x\in\R^n : |\langle \mathbf{c}_i, \x \rangle|\le 1 \mbox{ for all } 1\le i \le f\},
\end{equation*}
and that for all $\x \in \R^n$, we have
\begin{equation} \label{eq_poly}
	\|\x\|_N = \max_{1\le i \le f} |\langle \mathbf{c}_i, \x \rangle|.
\end{equation}

Consider a linear function $\varphi: \R^n \rightarrow \R^f$ defined for all $\x \in \R^n$ by
\begin{equation*}
	\varphi(\x) = \big(\langle \mathbf{c}_1, \x \rangle, \dots , \langle \mathbf{c}_f, \x \rangle\big).
\end{equation*}
It follows from~\eqref{eq_poly} that $\|\x\|_N = \|\varphi(\x)\|_\infty$ for all $\x \in \R^n$. Thus, $\varphi$ provides an isometric embedding of $\R_N^n$ into $\R_\infty^f$. Finally, it is easy to see that the two-coloring of $\R^f$ from Theorem~\ref{th_main} induced on the image $\varphi(\R^n)$ possesses the desired property.

\section{Concluding remarks} \label{sec_concl}

The statement of Theorem~\ref{th_main} raises the following problem. Given $n \in \N$, what is the minimum $k=k(\R_\infty^n) \le 5^n$ such that $\chi(\R_\infty^n,\B_k)=2$? With a more careful choice of the auxiliary parameters for the snake coloring, we can show that $k(\R_\infty^n) = O(3^n)$. However, this approach does not seem to lead to any subexponential upper bound.  On the other hand, it follows from~\eqref{eq_KS} that $k(\R_\infty^n)\ge n/\ln(2)$. It would be interesting to reduce the gap between these bounds and find the correct asymptotic.

The cases of small dimensions are usually of independent interest regarding the problems of this flavor. It is not hard to check that $\chi(\R_\infty^2,\B_1)=4$, $\chi(\R_\infty^2,\B_2)=\chi(\R_\infty^2,\B_3)=3$ and $\chi(\R_\infty^2,\B_k)=2$ for all $k \ge 4$. However, finding such a complete list seems to be a computationally challenging problem even in case of three-dimensional space $\R_\infty^3$.

Following the proof of Corollary~\ref{cor1}, one can easily show that the minimum $k=k(\R_p^n)$ such that $\chi(\R_p^n,\B_k)=2$ does not exceed $n\cdot5^n$ for all $1<p<\infty$. However, it is natural to conjecture that $k(\R_p^n)$ is bounded for such values of $p$ as $n \to \infty$ by analogy with the Euclidean case, where we have $k(\R_2^n) \le 5$ for all $n \in \N$. 

The situation with the Manhattan distance is also obscure for us. We know the way to establish an almost linear lower bound $k(\R_1^n) \ge n^{1-o(1)}$, but the best upper bound we have is only doubly exponential, see Corollary~\ref{cor2}.

Furthermore, note that the upper bound on $k(\R_N^n)$ from Corollary~\ref{cor2} depends only on the number of facets of the unit ball, and is independent of the dimension $n$. For instance, if $P(2f)$ is a norm on the plane whose unit ball is a regular $(2f)$-gon, then Corollary~\ref{cor2} implies that $k\big(\R_{P(2f)}^2\big) \le 5^f$. Is there a uniform upper bound independent of $f$ for this sequence of norms?

Finally, one can also consider the multicolor versions of all the aforementioned problems, where each isometric copy of $\B_k$ is required to intersect not with two, but with a greater number of colors.

\end{document}